\documentclass[11pt]{amsart}
\usepackage{mathrsfs}
\usepackage{amssymb}
\usepackage{amsmath,amscd}
\usepackage{amsfonts,amssymb}
\usepackage{latexsym}
\usepackage{amsbsy}
\usepackage{graphicx}
\usepackage{overpic}
\usepackage[all]{xy}

\newtheorem{theorem}{Theorem}
\newtheorem*{theorem*}{Theorem}

\newtheorem{lemma}{Lemma}
\newtheorem{proposition}{Proposition}

\newtheorem*{conjecture*}{Conjecture}

\newtheorem*{remark*}{Remark}

\newcommand{\bcen}{\begin{center}}      \newcommand{\ecen}{\end{center}}
\newcommand{\bay}{\begin{array}}      \newcommand{\eay}{\end{array}}
\newcommand{\beq}{\begin{eqnarray*}}      \newcommand{\eeq}{\end{eqnarray*}}

\def\ra{\rightarrow}

\def\dim{\operatorname{dim}}

\def\uq2{U_q(\hat{sl}_2)}

\def\nd{{\noindent}}
\def\mk{{\medskip}}

\title[Remark on theorem]{Remark on a theorem in Mumford's Red Book of Varieties and Schemes}

\thanks{}
\author{Guanglian Zhang}
\address{School of Mathematical Sciences\\
Shanghai Jiao Tong University\\
Shanghai 200240, China}
 \email{g.l.zhang@sjtu.edu.cn}

\date{\today}

\begin{document}

\begin{abstract}
In this paper, we firstly point out, by a counter example, that Proposition 6.4 of Section 6 in Bump's book (\cite{B}) is  error, and then give a correct statement with proof. We finally point out a gap in the proof of Theorem 3,  in Chapter I Section 8, of Mumford's red book \cite{M}, and indicate a way to complete it.

\end{abstract}

\maketitle

\bigskip
\bigskip

\section{Introduction}

In order to show \cite[Chapter I, \S~8, Theorem3]{M}, Bump divided his proof into four propositions in \cite{B}. But, there is a mistake in one of these propositions (\cite[Proposition 6.4]{B}). The original proof of Mumford is based on similar ideas, so we find a similar gap in the proof of \cite[Chapter I, \S~6, Theorem 3]{M}. In the following, we first give a counter example of the last statement in \cite[Proposition 6.4]{B}, then present a correct statement and prove it. Finally, we complete the proof of Theorem 3 in Mumford's red book \cite{M}.

\mk\nd{\bf Acknowledgments.} I would like to express my sincere gratitude to Jilong Tong for an interesting discussion.

\section{A counter example}

Let $k$ be an algebraically closed field. According to  ~\cite{B}, an algebraic set is a \emph{variety} if it is irreducible. The following proposition can be found in the book of Bump:

\begin{proposition}[\cite{B} Proposition 6.4]\label{p:1}~Let $\phi:X\rightarrow Y$ be a finite dominant morphism of affine varieties. Then $\phi$ is surjective. The fibers of $\phi$ are all finite. If $Z$ is a closed subset of $X,$
then $\phi(Z)$ is closed, and $\dim(Z)=\dim(\phi(Z)).$ If $W$ is closed subvariety of $Y$, and $Z$ is any irreducible component of $\phi^{-1}(W),$ then $\phi(Z)=W,$ and $\dim (Z)=\dim (W).$
\end{proposition}

The last part of  this proposition is wrong, and here is a counter example. Assume $k=\mathbb{C}$. Let $X=V(x^2+y^2+(z-\frac{1}{2})^2=\frac{1}{4})\subset \mathbb{A}^3$, and $Y$ the image of the morphism below
$$
\phi:X\longrightarrow \mathbb{A}^3, \quad (x,y,z)\longmapsto ((1-2z)x,(1-3z)y,(1-z)z).
$$
Let $W=\phi(X\cap V(z-\frac{1}{4}))$. Note that $X\cap V(z-\frac{1}{4})$ is an irreducible closed subset of $X$.

\noindent {\bf Claim:} $\phi$ is a finite morphism, and
$$
\phi^{-1}(W)=X\cap V\left(z-\frac{1}{4}\right)\bigcup\left\{\left(\frac{\sqrt{3}}{4},0,\frac{3}{4}\right), \left(-\frac{\sqrt{3}}{4},0,\frac{3}{4}\right)\right\}.
$$
In particular, $Z:=\{(-\frac{\sqrt{3}}{4},0,\frac{3}{4})\}$ is an irreducible component of $\phi^{-1}(W)$ such that $\phi(W)\neq Z$.
\begin{proof} Set $A=k[x,y,z]/(x^2+y^2+z^2-z)$. By abuse of notation, we shall use the same symbols to denote the images of $x,y,z\in k[x,y,z]$ in the quotient $A$. We first show that $\phi$ is a finite morphism. Consider the following morphism of $k$-algebras induced by $\phi$
$$
\lambda:k[a,b,c]\rightarrow A, \quad (a,b,c)\mapsto ((1-2z)x,(1-3z)y,(1-z)z),
$$
which makes $A$ an algebra over $k[a,b,c]$. We need show that $A$ is integral over $k[a,b,c]$. It is clear that $z$ is integral over $k[a,b,c]$ since $z^2-z+\lambda(c)=0$ by the definition of $\lambda$. Moreover, as $x^2+y^2+z^2-z=0$ in $A$, to see that $A$ is integral over $k[a,b,c]$, it suffices to show that $y$ is integral over the $k$-subalgebra $k[\lambda(a),\lambda(b),\lambda(c),z]\subseteq A$ of $A$ generated by $\lambda(a),\lambda(b),\lambda(c),z\in A$. We shall do this by finding an integral relation for it. First, by the definition of $\lambda$, we have
$$
zx =\frac{x-\lambda(a)}{2},\quad zy  =\frac{y-\lambda(b)}{3},\quad \textrm{and} \quad
x^2+y^2 =\lambda(c),
$$
giving
\begin{eqnarray*}
\lambda(b)^2 & =& [(1-3z)y]^2=y^2-6zy^2+9z^2y^2\\
&= & y^2-6zy^2+9(z^2\lambda(c)-z^2x^2)\\
& = & y^2-6y\frac{y-\lambda(b)}{3}-9(\frac{x-\lambda(a)}{2})^2+9z^2\lambda(c)\\
& = & -y^2+2\lambda(b)y-\frac{9}{4}x^2+\frac{9}{2}\lambda(a)x-\frac{9}{4}\lambda(a)^2+9z^2\lambda(c)\\
&=& -y^2-\frac{9}{4}(\lambda(c)-y^2)+2\lambda(b)y+\frac{9}{2}\lambda(a)x-\frac{9}{4}\lambda(a)^2+9z^2\lambda(c)\\
&=& \frac{5}{4}y^2+2\lambda(b)y+\frac{9}{2}\lambda(a)x-\frac{9}{4}\lambda(a)^2+9z^2\lambda(c)-\frac{9}{4}\lambda(c).
\end{eqnarray*}
Let $w=-\frac{9}{4}\lambda(a)^2+9z^2\lambda(c)-\frac{9}{4}\lambda(c)\in k[\lambda(a),\lambda(b),\lambda(c),z]\subseteq A$. Consequently,
\begin{eqnarray*}
(1-2z)\lambda(b)^2 & = & (1-2z)[\frac{5}{4}y^2+2\lambda(b)y+\frac{9}{2}\lambda(a)x+w]\\
& = & \frac{5}{4}y^2(1-2z)+2\lambda(b)(1-2z)y+\frac{9}{2}\lambda(a)^2+(1-2z)w \\ & = & \frac{5}{4}y^2-\frac{5}{2}zy^2+2\lambda(b)(1-2z)y+\frac{9}{2}\lambda(a)^2+(1-2z)w \\
& =& \frac{5}{4}y^2-\frac{5}{2}y\frac{y-\lambda(b)}{3}+2\lambda(b)(1-2z)y+\frac{9}{2}\lambda(a)^2+(1-2z)w \\
& =& \frac{5}{12}y^2+\frac{5}{6}\lambda(b)y+2\lambda(b)(1-2z)y+\frac{9}{2}\lambda(a)^2+(1-2z)w.
\end{eqnarray*}
In particular, we obtain the following equality in $A$:
$$
\frac{5}{12}y^2+\frac{5}{6}\lambda(b)y+2\lambda(b)(1-2z)y+\frac{9}{2}\lambda(a)^2+(1-2z)w-(1-2z)\lambda(b)^2 =0.
$$
As $w\in k[\lambda(a),\lambda(b),\lambda(c),z]\subseteq A$, we deduce that $y$ is integral over $k[\lambda(a),\lambda(b),\lambda(c),z]$, thus also integral over $k[a,b,c]$. Therefore, $\phi$ is finite, as claimed.

We now determine $\phi^{-1}(W)$ by computing the fibers of $\phi$. For $i=1,2$, let $(x_i,y_i,z_i)\in X\subset \mathbb A^3$ such that $\phi(x_1,y_1,z_1)=\phi(x_2,y_2,z_2)$. So
\begin{eqnarray*}
\begin{cases} (1-2z_1)x_1=(1-2z_2)x_2\\ (1-3z_1)y_1=(1-3z_2)y_2\\ (1-z_1)z_1=(1-z_2)z_2=:-a.
\end{cases}
\end{eqnarray*}
In particular, $z_1,z_2$ are roots of $z^2-z-a=0$, and
 \begin{eqnarray*}
\begin{cases} x_1^2+y_1^2+a=0\\ x_2^2+y_2^2+a=0.
\end{cases}
\end{eqnarray*}
We shall distinguish the following four different cases:
\begin{itemize}
\item Case 1: $z_1=z_2=\frac{1}{3}$. We have
\begin{eqnarray*}
\begin{cases} x_1=x_2\\y_1=\pm y_2\\
z_1=z_2=\frac{1}{3}.
\end{cases}
\end{eqnarray*}

\item Case 2: $z_1=z_2=\frac{1}{2}$. We have
\begin{eqnarray*}
\begin{cases} x_1=\pm x_2\\y_1= y_2\\
z_1=z_2=\frac{1}{2}.
\end{cases}
\end{eqnarray*}

\item Case 3: $z_1=z_2\notin \{\frac{1}{2}, \frac{1}{3}\}$. We have
\begin{eqnarray*}
\begin{cases} x_1= x_2\\y_1= y_2\\
z_1=z_2.
\end{cases}
\end{eqnarray*}

\item Case 4: $z_1\neq z_2$. Then $z_1,z_2$ are the two roots of $z^2-z-a=0$. Consequently,
\begin{eqnarray*}
\begin{cases} z_1+z_2=1\\z_1z_2=-a,
\end{cases}
\end{eqnarray*}
and
\begin{eqnarray*}
\begin{cases} (1-2z_1)x_1=[1-2(1-z_1)]x_2=(2z_1-1)x_2,\\ (1-3z_1)y_1=[1-3(1-z_1)]y_2=(3z_1-2)y_2\\ (1-z_1)z_1=(1-z_2)z_2
\end{cases}\\
\Rightarrow \begin{cases} x_1=-x_2{\quad\quad\quad\quad\quad\quad}\\ y_1^2=y_2^2{\quad\quad\quad\quad\quad\quad}\\
(1-3z_1)y_1=(3z_1-2)y_2{\quad\quad\quad\quad\quad\quad}\\
(1-z_1)z_1=(1-z_2)z_2=-a.{\quad\quad\quad\quad\quad\quad}
\end{cases}
\end{eqnarray*}
In particular, $y_1=\pm y_2$. If $y_1=y_2\neq 0$, we have $1-3z_1=3z_1-2,$ thus $z_1=z_2=\frac{1}{2}$, which is impossible. If $y_1=-y_2\neq 0,$ we have $1-3z_1=2-3z_1$, giving also a
contradiction. So, we must have $y_1=y_2=0$ in this case. Hence
\begin{eqnarray*}
\begin{cases} x_1=-x_2\\y_1=y_2=0\\z_1+z_2=1, \text{~and~} z_1\neq z_2.
\end{cases}
\end{eqnarray*}
\end{itemize}

Based on the above discussion, we deduce
$$\phi^{-1}(W)=X\cap V(z-\frac{1}{4})\bigcup \left\{\left(\frac{\sqrt{3}}{4},0,\frac{3}{4}\right), \left(-\frac{\sqrt{3}}{4},0,\frac{3}{4}\right)\right\},
$$
and the last statement then follows easily. This completes the proof of our claim.
\end{proof}

We now give a corrected form of the last part of Proposition \ref{p:1}.

\begin{proposition}\label{p:2} Let $\phi:X\rightarrow Y$ be a finite dominant morphism of affine varieties. Assume $Y$ is \emph{normal}. Let $W$ be an irreducible closed subvariety of $Y$, and $Z$ an irreducible component of $\phi^{-1}(W)$. Then $\phi(Z)=W$, and $\dim (Z)=\dim (W)$.
\end{proposition}

\begin{proof} Let $A=k[X]$ and $B=k[Y]$. As $X$ and $Y$ are affine varieties, one can identify $B$ as a $k$-subalgebra of $A$ using the dominant morphism $\phi$. Write $ W=V(\mathfrak q)$, with $\mathfrak q\subset B$ a prime ideal. So $\phi^{-1}(W)=V(\mathfrak qA)$. Let $\widetilde{\mathfrak q}_1,\widetilde{\mathfrak q}_2,\cdots,\widetilde{\mathfrak q}_r$ be the minimal prime ideals of $V(\mathfrak qA)$. Then $V(\widetilde{\mathfrak q}_1),V(\widetilde{\mathfrak q}_2),\cdot\cdot\cdot,V(\widetilde{\mathfrak q}_r)$ are the irreducible components of $\phi^{-1}(W)$. Since $A$ is integral over $B$,  thanks to \cite[Section 1, Proposition 4.3]{B}, we have $\mathfrak qA\cap B=\mathfrak q$. We now claim that $\mathfrak q=\widetilde{\mathfrak q}_i\cap B$ for all $i$. Clearly $\widetilde{\mathfrak q}_i\cap B\supseteq \mathfrak qA\cap B= \mathfrak q$. Suppose there exists some $i$ such that $\widetilde{\mathfrak q}_i\cap B\supsetneq \mathfrak q$. Because $Y$ is normal, by going-down theorem, there exists a prime ideal ${\widetilde{\mathfrak q}'_i}$ of $A$ such that $\widetilde{\mathfrak q}'_i\subsetneqq \widetilde{\mathfrak q}_i, $ and ${\widetilde{\mathfrak q}'_i}\cap B=\mathfrak q$. In particular, $\widetilde{q}_i\supsetneq \widetilde{q}'_i\supseteq qA$. But this contradicts to the fact that $\widetilde{\mathfrak q}_i\in V(\mathfrak qA)$ is a minimal ideal, proving our claim. Consequently, $\phi$ maps the generic point of $Z$ to that of $W$. So, by the second part of Proposition \ref{p:1}, we find $\phi(Z)=W$ and $\dim(Z)=\dim(W)$.





\end{proof}

\section{The proof}
We now in the position to complete the proof of Theorem 3 at Section 6 of Chapter 1 of Mumford's red book. First of all, we need a lemma, which is a special case of a more general well-known statement. 








\begin{lemma}\label{L:1}  Let $A$ be a $k$-algebra of finite type. Assume $A$ is a domain. Then, there exists some $f\in A\setminus \{0\}$, such that the localisation $A_f$ is normal.
\end{lemma}

\begin{proof} Let $K$ denote the fraction field of $A$, and $A'$ the integral closure of $A$ in $K$. Since $A$ is a finitely generated over a field, $A'$ is finite as an $A$-module by \cite[Chapter 12 Theorem 72]{Mat}. In particular, there exist $x_1,\cdots,x_n\in A'$ with $A'=\sum_{i=1}^n A\cdot x_i\subseteq K$.  As $K$ is the fraction field of $A$, one can find $f\in A\setminus \{0\}$ such that $fx_i\in A$ for all $i$. Therefore, $A[1/f]\subset A'[1/f]=\sum_{i=1}^{r}A[1/f]x_i\subset A[1/f]\subseteq K$. Thus, $A_f=A[1/f]=A'[1/f]=A'_f$ is the localisation of the normal ring $A'$, hence is normal as well.
\end{proof}

\begin{theorem}[\cite{M} Chapter I \S~6 Theorem 3]\label{t:1} Let $f:X\rightarrow Y$ be a dominating morphism of varieties and let $r=\dim X-\dim Y$. Then there exists a nonempty open set $U\subset Y$ such that:
\begin{enumerate}
\item $U\subset f(X)$, and
\item for all irreducible closed subsets $W\subset Y$ such that $W\cap U\neq\emptyset,$ and for all irreducible components $Z$ of $f^{-1}(W)$ such that $Z\cap f^{-1}(U)\neq\emptyset$,
$$\dim Z=\dim W+r$$
or
$$\mathrm{codim}(Z, X)=\mathrm{codim}(W,Y).$$
\end{enumerate}
\end{theorem}

\begin{proof} As in the original proof of Mumford, we reduce to the following case: $X,Y$ are affine, and there exists some non-empty open subset $U\subset Y$, such that the induced map $f^{-1}(U)\ra U$ is decomposed as
\[
f^{-1}(U)\stackrel{\pi}{\longrightarrow}U\times \mathbb A^r \longrightarrow U,
\]
where the first map $\pi$ is finite and dominant, while the second is the natural projection. So $\pi$ is surjective by the first part of Proposition \ref{p:1}, and $U\subset f(X)$. Shrinking $U$ if necessary, we further assume $U$ normal according to Lemma \ref{L:1} above. In particular, the affine variety $U\times \mathbb A^r$ is also normal. To finish the proof, let $W\subset Y$ be an irreducible closed subset that meets $U$, and let $Z\subset X$ be an irreducible component of $f^{-1}(W)$ such that $Z\cap f^{-1}(U)\neq \emptyset$. Let $W_0=W\cap U$, and $Z_0=Z\cap f^{-1}(U)$. Then $\dim(W)=\dim(W_0)$ and $\dim(Z)=\dim(Z_0)$. Since $W_0$ is an irreducible closed subset of $U$, one checks that $W_0\times \mathbb A^r$ is an irreducible closed subset of $U\times \mathbb A^r$. Moreover, $Z_0$ is an irreducible component of $\pi^{-1}(W_0\times \mathbb A^r)=f^{-1}(W_0)$. As $W_0\times \mathbb A^r$ is normal, by Proposition \ref{p:2}, $\pi(Z_0)=W_0\times \mathbb A^r$ and $\dim(Z_0)=\dim(W_0\times \mathbb A^r)=\dim(W_0)+r$. Consequently, $\dim(Z)=\dim(W)+r$, as claimed by (2).

\end{proof}

\end{document}